\documentclass[10pt]{article}
\usepackage{amssymb,amsmath,amsthm,wasysym}
\usepackage{fullpage,graphicx,subfigure,color}
\newtheorem{theorem}{Theorem}
\newtheorem{lemma}[theorem]{Lemma}
\newtheorem{corollary}[theorem]{Corollary}

\newtheorem{remark}[theorem]{Remark}

\title{Zykov sums of digraphs with diachromatic number equal to their harmonious chromatic number\footnotemark[4]}
\author{Mika Olsen\footnotemark[1] \and Christian Rubio-Montiel\footnotemark[2] \and Alejandra Silva-Ram{\' i}rez\footnotemark[3]}

\begin{document}
\maketitle

\def\thefootnote{\fnsymbol{footnote}}
\footnotetext[1]{Departamento de Matem{\' a}ticas Aplicadas y Sistemas, UAM-Cuajimalpa, Mexico City, Mexico. {\tt
		olsen@correo.cua.uam.mx}.}
\footnotetext[2]{Divisi{\' o}n de Matem{\' a}ticas e Ingenier{\' i}a, FES Acatl{\' a}n, Universidad Nacional Aut{\'o}noma de M{\' e}xico, Naucalpan, Mexico. {\tt christian.rubio@acatlan.unam.mx}.}
\footnotetext[3]{Posgrado de Ciencias Naturales e Ingenier\'ia, UAM-Cuajimalpa, Mexico City, Mexico.   {\tt 2173800481@cua.uam.mx}.}
\footnotetext[4]{Supported by CONACYT of Mexico under projects A1-S-12891 and 47510664, and PAIDI of Mexico under project 007/21.} 	
\begin{abstract} 
The dichromatic number and the diachromatic number are generalizations of the chromatic number and the achromatic number  for digraphs considering acyclic colorings. In this paper, we determine the diachromatic number of digraphs arising from the Zykov sum of digraphs  that admit a complete $k$-coloring with $k=\tfrac{1+\sqrt{1+4m}}{2}$ for a suitable $m$. Consequently, the diachromatic number equals the harmonious number for every digraph in this family. In particular, we study the chromatic number, the diachromatic number, and the harmonious chromatic number of the Zykov sum of cycles. 
\end{abstract}
	
\textbf{Keywords.} Dichromatic number, factorization, detachments, composition of digraphs, lexicographic product.



\section{Introduction}

A coloring of a digraph can be understood as a function that maps elements of a graph, usually vertices, into some set, usually numbers, which are called colors and such that some property is satisfied, usually related to the arcs.  More precisely, a \emph{$k$-coloring} of a digraph $D$ is a proper vertex-coloring, that is, each chromatic class induces a subdigraph with no arcs.  The \emph{chromatic number} $\chi(D)$ of $D$ is the smallest $k$ for which there exists a $k$-coloring of $D$ \cite{MR2472389}.  The original concept of the chromatic number comes from graphs and then extended to digraphs in this natural way, as well as the following colorings and parameters.

A coloring of $D$ is called \emph{harmonious} if for every ordered pair $(i,j)$ of different colors there is at most one arc $uv$ such that $u$ is colored $i$ and $v$ is colored $j$. The \emph{harmonious chromatic number} $h(D)$ of $D$ is the smallest $k$ for which there exists a harmonious $k$-coloring of $D$ \cite{MR2895432,MR3329642}. 
A coloring of $D$ is \emph{complete} if for every ordered pair $(i,j)$ of different colors there is at least one arc $uv$ such that $u$ is colored $i$ and $v$ is colored $j$ \cite{MR2998438}.  The \emph{achromatic number} $\psi(D)$ of  $D$ is the largest $k$ for which there is a complete $k$-coloring of $D$ \cite{MR2998438}. Therefore, the size $m$ of a digraph $D$ is bounded below by $2\tbinom{\text{dac}(D)}{2}$, hence, $\psi(D)$ is bounded above by $\tfrac{1+\sqrt{1+4m}}{2}$ and both coincide if and only if there are exactly two arcs between every two chromatic classes.  For graphs (which can be seen as symmetric digraphs), such parameters are called the \emph{chromatic number} $\chi$, the \emph{harmonious chromatic number} $h$ and the \emph{achromatic number} $\psi$, respectively.
Edwards \cite{MR2998438} proved that determining the exact value of the harmonious chromatic number is NP-hard for digraphs of bounded degree and that for a given digraph the existence of a complete coloring is NP-complete. For this reason, we use the following generalization for the chromatic and the achromatic numbers of a digraph.

A vertex-coloring of a digraph $D$ is \emph{acyclic} if each chromatic class induces a subdigraph with no directed cycles. The \emph{dichromatic number} $\text{dc}(D)$ of $D$ is the smallest $k$ for which there exists an acyclic coloring of $D$ using $k$ colors \cite{MR693366}. Any $\text{dc}(D)$-coloring of $D$ is also complete. The \emph{diachromatic number} $\text{dac}(D)$ of  $D$ is the largest $k$ for which there is an acyclic and complete coloring of $D$ using $k$ colors \cite{MR3875016}. 
Therefore, for any digraph $D$ of size $m$,  we have that
\begin{equation} \label{eq1}
\text{dc}(D)\leq \text{dac}(D)\leq \tfrac{1+\sqrt{1+4m}}{2} \leq h(D).
\end{equation}

On the other hand, the study of parameters arising of complete colorings into graph products can be found in \cite{MR2008428,MR1196112,MR1189852,MR3265992,vijayalakshmi2017achromatic} with results for the cartesian product or join of graphs.
Let $D$ be a digraph and $X = \{H_u\colon u\in V(D)\}$ a family of nonempty mutually vertex-disjoint digraphs. The \emph{Zykov sum} $D[X]$ of $X$ over $D$ is a digraph with vertex set $\underset{u\in V(D)}{\bigcup}V(H_{u})$ and arc set \[\underset{u\in V(D)}{\bigcup}A(H_{u})\cup\left\{ ab:a\in V(H_{u}),b\in V(H_{v}),uv\in A(D)\right\} .\]

The corresponding operation for graphs is called \emph{generalized composition}. If $H_u\cong  H$ for every $u\in V(D)$, then $D[X]$ is called \emph{lexicographic product}  (also called \emph{digraph composition})  and is denoted by $D[H]$.  
It is known that the complexity of testing whether an arbitrary graph can be written nontrivially as the composition of two smaller graphs is the same as the complexity of testing whether two graphs are isomorphic \cite{MR837609} which can be solved in quasipolynomial time according to \cite{MR3536606}.
The dichromatic number of Zykov sums and composition of digraphs were studied in \cite{MR3711038,MR1817491} and the chromatic number of the lexicographic product of graphs were studied in \cite{MR1303390,MR392645}. 

In this paper, we determine digraphs, arising from the Zykov sum of digraphs, which accept a complete $k$-coloring with $k$ equals their harmonious number.  As a consequence,  we obtain results about graphs, arising from the lexicographic product of graphs,  with achromatic number equals their harmonious number.  Then, we analyze conditions to apply the results with particular attention to factorizations of the complete graphs into Hamiltonian cycles.  Also, we study a recursive application of the results.
	

\section{Main result}\label{section2}

In this paper, we consider only finite simple digraphs.  Let $[n]$ denote the set $\{1,2,\dots,n\}$ and let  $m\geq 2$. For the case of digraphs, $K_m$  denotes  the  complete symmetric digraph and for graphs, $K_m$  denotes  the  complete  graph. A \emph{factor} $H_j$ of the complete digraph (resp. graph) $K_m$ is a spanning subdigraph (resp. subgraph). A \emph{factorization} $Y$ of $K_m$ is a set of $q$ pairwise arc-disjoint (resp. edge-disjoint) factors $H_j$,  for $j\in [q]$, such that these factors induce a partition of the arcs (resp. edge) of $K_m$.  If $H_j \cong H$ (for all $j\in [q]$) then it is called an \emph{$H$-factorization}.  Given a factorization $Y,$ a \emph{relabel factorization} $X$ of $Y$ is a relabeling of the vertices of $Y$ in the following way: the vertices $\{v^1,v^2,\dots,v^m\}$ of the factor $H_j$ is relabeled into $\{v^1_j,v^2_j,\dots,v^m_j\}$. 

Let $D$ be a $k$-diachromatic digraph  (resp. $k$-achromatic graph)  with a $k$-coloring $\varphi$. Let $\{V_1,V_2,\dots,V_k\}$ be the set of chromatic classes for $\varphi$ with $|V_i|=q_i$. For each $i\in [k]$, denote the vertices of the chromatic class $V_i$ by $\{u_{i,1},u_{i,2},\dots u_{i,q_i}\}$. In this case $V(D)=\bigcup_{i=1}^k V_i$. 
For each $i\in [k]$, let $X_i=\{H_{u_{i,1}},H_{u_{i,2}},\dots,H_{u_{i,{q_i}}}\}$ be a relabel factorization of $K_{m_i}$ into $q_i$ factors
. We consider the Zykov sum $D[X]$, where $X=\underset{i=1}{\overset{k}{\bigcup}}X_{i}=\{H_{u_{i,j}}\colon u_{i,j}\in V(D)\}$. 

Before proving our first theorem, we require the following result and definitions.  For two nonempty vertex sets $V_1,V_2$  of a digraph $D$, we define $[V_1,V_2] = \{ (x,y)\in A(D) \mid  x \in V_1, y \in V_2 \}.$ A digraph $D$ is \emph{$k$-minimal} if $\text{dac}(D)=k$ and $\text{dac}(D-f)<k$ for all $f\in A(D)$.

\begin{theorem} \cite{MR3875016} \label{teo1}
Let $D$ be a digraph with diachromatic number $k$. Then, $D$ is $k$-minimal if and only if $D$ has size $k(k-1)$.
\end{theorem}
	
\begin{theorem}\label{teo2}
Let $D$ be a $k$-minimal digraph of order $n$ with a $k$-coloring $\varphi$. Let $\{V_1,V_2,\dots,V_k\}$ be the set of chromatic classes for $\varphi$ with $|V_i|=q_i$. For each $i\in [k]$ let $V_i=\{u_{i,1},u_{i,2},\dots u_{i,q_i}\}$ and let $X_i=\{H_{u_{i,1}},H_{u_{i,2}},\dots,H_{u_{i,{q_i}}}\}$ be relabel factorizations of $K_{m_i}$ into $q_i$ factors. Then $D[X]$ is $t$-minimal, where $X=\underset{i=1}{\overset{k}{\bigcup}}X_{i}$ and $t=\overset{k}{\underset{i=1}{\sum}}m_{i}$.
\end{theorem}
\begin{proof}
We take a partition of $K_{m_i}$ into $q_i$ factors. In order to have a set of colored and sorted vertices arising from $V(K_{m_i})=\{v^1_{u_i},v^2_{u_i},\dots,v^{m_i}_{u_i}\}$, we define the following coloring. Let $f_i\colon V(K_{m_i})\rightarrow [m_i]$ be the complete $m_i$-colorings of $K_{m_i} $ such that $f_i(v_{u_i}^l)=l$ for each $l\in[m_i]$. Let $f_{i,j}\colon V(H_{u_{i,j}})\rightarrow [m_i]$ the natural restriction of $f_i$ into each factor $H_{u_{i,j}}$, that is, $f_{i,j}(v^l_{u_{i,j}})=f_i(v^l_{u_i})=l$ for any vertex $v^l_{u_{i,j}}\in V(H_{u_{i,j}})$, with $i\in [k]$, $j\in [q_i]$ and $l\in [m_i]$, see Figure \ref{Fig1} for an example.
		
Let $\varsigma\colon V(D[X])\rightarrow [t]$ be a $t$-coloring such that for each  $l\in[m_i]$ \[\varsigma(v_{u_{i,j}}^l)=c(i,l):=\overset{i-1}{\underset{a=0}{\sum}}m_{a}+l, \mbox{with }  m_0=0.\]
That is, if $i$ and $l$ are fixed,   for each $j\in [q_i$] the  vertex $v_{u_{i,j}}^l$ in  the factor $H_{{u_{i,j}}}$ has   color $c(i,l)$. Thus, the set of vertices colored $c(i,l)$ of $\varsigma$ is \[\{v^l_{u_{i,1}},v^l_{u_{i,2}},\dots,v^l_{u_{i,q_i}}\}.\]
Since the Zykov sums of empty graphs is empty, the coloring is proper and then acyclic due to the fact that the induced subgraph by $\{v^l_{u_{i,1}},v^l_{u_{i,2}},\dots,v^l_{u_{i,q_i}}\}$ of $D[X]$ is empty.

Next, we claim the $\varsigma$ coloring is minimal and complete. Let $c(i,l)$ and $c(i',l')$ be two colors of $\varsigma$ with $i, i'\in [q_i]$, $l\in[m_i]$ and $l'\in [m_{i'}]$. If $i=i'$, since each $H_{u_{i,j}}$ has the $m_i$ colors of $f_i$, then $v^l_{u_{i,j}}v^{l'}_{u_{i,j}}$ is the unique arc of $H_{u_{i,j}}$ for some $j$ and then there exists a unique arc between $c(i,l)$ and $c(i,l')$. On the other hand, since $\varphi$ is minimal and complete, if $i\not =i'$ there exists a unique arc $u_{i,j}u_{i',j'}$ such that $\varphi (u_{i,j})=i$ and $\varphi (u_{i',j'})=i'$ with $j\in [q_i]$ and $j'\in [q_i']$. Therefore, $\left[V(H_{{u_i},j}),V(H_{{u_i'},j'})\right]$ is a bipartition of a directed complete bipartite subdigraph of $D[X]$. In consequence,  for  a fixed $l$ and $l'$ the arc $v^l_{u_{i,j}}v^{l'}_{u_{i',j'}}$ is  the unique arc from a vertex of color $c(i,l)$ to a vertex with color $c(i',l')$.
\end{proof}

Figure \ref{Fig1} shows the example of the Zykov sum $\protect\overrightarrow{C}_{6}[X]$ for some set of digraphs $X$, where $X=\{X_1, X_2, X_3\}$  and  $X_1=\{H_{u_{1,1}},H_{u_{1,2}}\}$, $X_2=\{H_{u_{2,1}},H_{u_{2,2}}\}$ and $X_3=\{H_{u_{3,1}},H_{u_{3,2}}\}$ are relabel factorizations of $K_2$, $K_3$ and $K_4$, respectively.  In order to avoid drawing all the arc between subdigraphs we use the symbol $\Rightarrow$. An edge represents a couple of symmetric arcs.

\begin{figure}[htbp!]	
\begin{center}	
\includegraphics{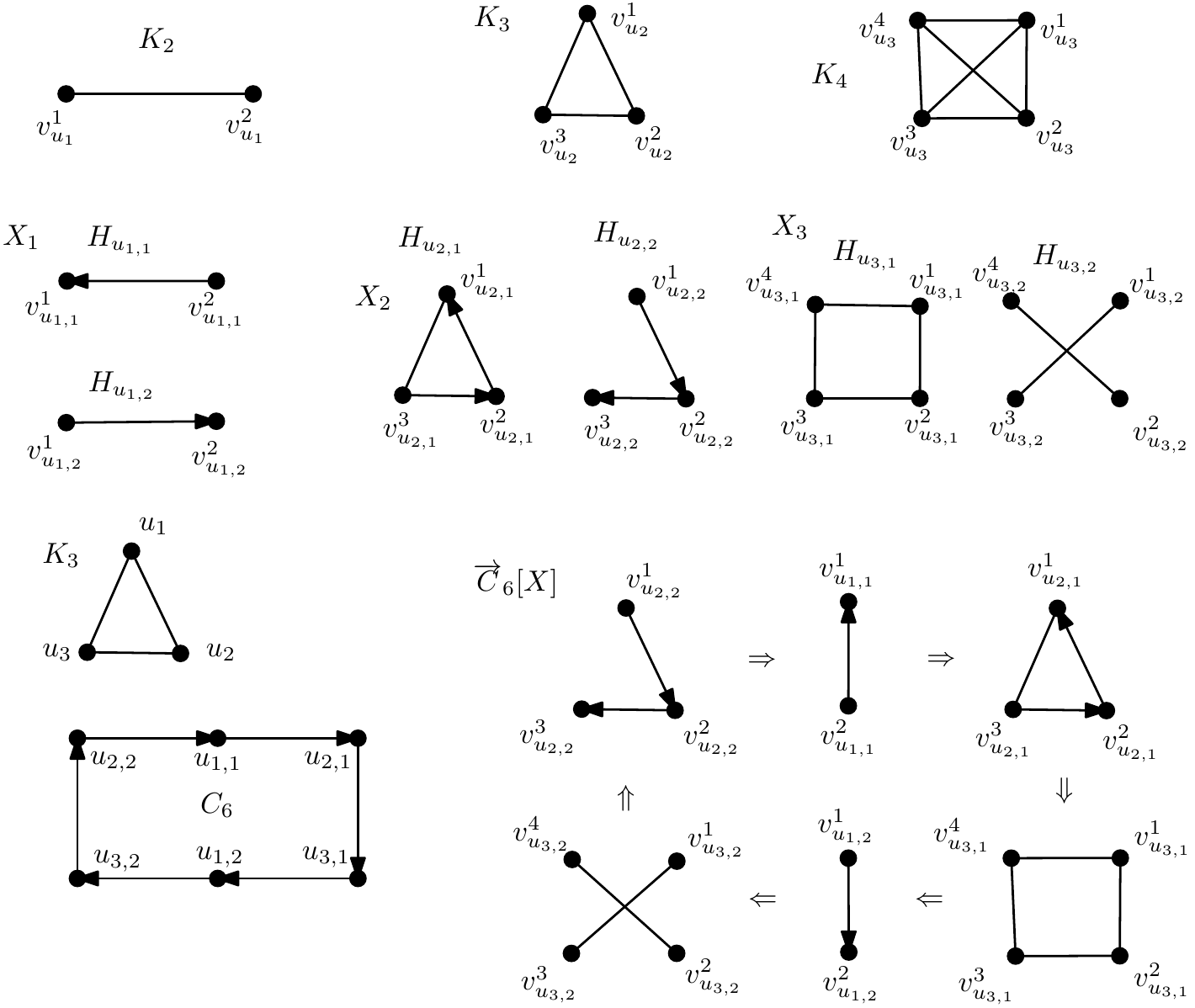}
\caption{\label{Fig1} The Zykov sum $\protect\overrightarrow{C}_{6}[X]$ where $X$ is a relabel factorization of $K_2$, $K_3$ and $K_4$. }
\end{center}
\end{figure}

Now, we have the following corollaries due to Theorem \ref{teo2}.  For this, we recall 
that a \emph{balanced coloring} is a coloring in such a way that any two chromatic classes have the same cardinality. 

\begin{corollary}\label{cor3} 
Let $D$ be a $k$-minimal digraph of order $n$ with a balanced  $k$-coloring such that $qk=n$. 
Let $X_i$ be relabel factorizations of $K_m$ into $q$ factors, that is, $X_i=\{H_{u_{i,1}},H_{u_{i,2}},\dots,H_{u_{i,q}}\}$ for $i\in [k]$. Then $D[X]$ is $km$-minimal with a balanced  $km$-coloring where $X=\underset{i=1}{\overset{k}{\bigcup}}X_{i}$.
\end{corollary}

\begin{corollary}\label{cor4}
Let $D$ be a $k$-minimal digraph of order $n$ with a balanced $k$-coloring such that $qk=n$.
If $K_m$ has a relabel $H$-factorization into $q$ factors, then $D[H]$ is $km$-minimal with a balanced $km$-coloring.
\end{corollary}

\section{The lexicographic product of cycles}\label{section3}

Now, we proceed to construct families of digraphs obtained  by  Zykov sums  $D$ and $H$ that satisfy the hypothesis of Theorem \ref{teo2}, and then we give results for $\overrightarrow{C}_{m}[\overrightarrow{C}_{n}]$ for some values of $m$. We recall some definitions given in \cite{MR3875016}. 

Let $u$ and $v$ be two vertices of a digraph $D$ such that  $uv$ is an arc of $D$.  We say that $u$ is \emph{incident to} $v$ and $v$ is \emph{incident from} $u$.  The out-neighborhood $N ^+(u)$ of a vertex $u$ is the set of vertices that are incident from $u$. Similarly, the in-neighborhood $N^-(v)$ of a vertex $v$ is the set of vertices incident to $v$. 

Two vertices are \emph{adjacent} if they are in a $2$-cycle. To obtain an \emph{elementary dihomomorphism} of a digraph $D$, identify two nonadjacent vertices $u$ and $v$ of $D$. The resulting vertex, when identifying $u$ and $v$, may be denoted by either $u$ or $v$.  A sequence of elementary dihomomorphisms is a \emph{dihomomorphism}. For graphs, a dihomomorphism image corresponds to a usual homomorphism image also called \emph{amalgamation}, see \cite{MR916377}. A graph $G$ is an amalgamation of a graph $H$ if and only if $H$ is a \emph{detachment} of $G$, see \cite{MR2138114}.

An elementary dihomomorphism preserving the cardinality of arcs is called \emph{elementary identification} $\epsilon$, that is, let $D$ be a digraph and $u,v\in V(D)$ two independent vertices such that $N^+(u)\cap N^+(v)=\emptyset$ and $N^-(u)\cap N^-(v)=\emptyset$, then $\epsilon$ is the elementary dihomomorphism obtained by identifying $u$ and $v$. 
A digraph $D'$ is an  identification image of a digraph $D$ if and only if $D'$ can be obtained by a sequence of elementary identifications beginning with $D$.
	
An \emph{elementary unfold} is the inverse of an elementary identification and an \emph{unfold} is the inverse of an identification.  For graphs,  an \emph{exact graph} has exactly $\binom{k}{2}$ edges for some integer $k$, see \cite{MR2138114}. Then, we say a digraph is \emph{exact} if it has exactly $k(k-1)$ arcs for some integer $k$, and we call an identification image as an \emph{exact amalgamation} and an unfold image as an \emph{exact detachment}. For example, an exact detachment of  $K_5$ is $\overrightarrow{C}_{20}$ if we follow a Eulerian circuit of $K_5$, and vice versa, an exact amalgamation of $\overrightarrow{C}_{20}$ is  $K_5$.  Also, it is clear that a factorization of a digraph can be understood as an exact detachment. 

\begin{remark}\label{identificacion}
A digraph $D$ is $k$-minimal if and only if there exists an  identification $\Gamma$  from the digraph $D$ to the complete digraph $K_k$.  
\end{remark}

As a consequence of  Remark \ref{identificacion},  we have that $ K_k $ can be unfolded in the cycle $\overrightarrow{C}_{k(k-1)}
$. The induced $k$-coloring of $\overrightarrow{C}_{k(k-1)}$ is balanced where each chromatic class has $k-1$ vertices. On the other hand, $K_k$ accepts a $\overrightarrow{C}_{k}$-factorization into $k-1$ factors for $k\not=4,6$, see \cite{MR505807,MR584161}. 

\begin{corollary}\label{cor6}
Let $n$ be a natural number such that $n\not=4,6$. The digraph $D=\overrightarrow{C}_{n^2-n}[\overrightarrow{C}_n]$ is $n^2$-minimal with a balanced $n^2$-coloring. Hence
\[\text{\emph{dac}}(D)=h(D)=n^2.\]
\end{corollary}
\begin{proof}
Since, $K_n$ can be unfold into $n-1$ digraphs isomorphic to $\overrightarrow{C}_{n}$, and $K_{n}$ can be unfold into $\overrightarrow{C}_{n^2-n}$ with a balance $n$-coloring where each chromatic class has $n-1$ vertices. Therefore,  $\overrightarrow{C}_{n^2-n}[\overrightarrow{C}_{n}]$ is $n^2$-minimal and the result follows.
\end{proof}

\subsection{On the diachromatic number}

In this subsection, we bound the diachromatic number of $D=\overrightarrow{C}_{m}[\overrightarrow{C}_{n}]$ for $m$ close to $n^2-n$.

To begin with, we improve the bound in Equation \ref{eq1} for $m=n^2-n+t$ following the idea of comparing two functions one of which determines the maximum possible number of chromatic classes of order $x$ and the other one determines how many chromatic classes can be incident to one chromatic class of order $x$, where $x$ is the order of the smallest chromatic class.

\begin{theorem}\label{teo7}
Let $m,n$ be natural numbers such that $m,n\geq 3$. For any complete coloring of $D=\overrightarrow{C}_{m}[\overrightarrow{C}_{n}]$ using $k$ colors
\[k\leq \max\left\{ \min\{f_n(x),g_n(x)\} \textrm{ with } x\in\mathbb{N}\right\}\] where $f_n(x)=\left\lfloor mn/x\right\rfloor$ and $g_n(x)=x(n+1)+1$.
\end{theorem}
\begin{proof}
Let $\varsigma\colon V(D) \rightarrow [k]$ be a complete $k$-coloring of $D$.  Let $x= \min\{ \left|\varsigma^{-1}(i)\right| : i\in \left[k\right]\}$ be the cardinality of the smallest chromatic class of $\varsigma$. Without loss of generality,  suppose that $x= \left|\varsigma^{-1}(k)\right|$. Since $\varsigma$ defines a partition of $V(D)$, it follows that $k\leq mn/x$ and then $k\leq f_n(x)$.

On the other hand, there are $x(n+1)$ vertices from any smallest chromatic class,  then there are at most $x(n+1)+1$ chromatic classes, and it follows that $k\leq g_n(x)$. Thus, we have that $k\leq\min\{f_n(x),g_n(x)\}.$
Finally, we obtain that
\[k\leq \max\left\{ \min\{f_n(x),g_n(x)\} \textrm{ with } x\in\mathbb{N}\right\}.\]
\end{proof}

Since $f_n(x)$ is a hyperbola and $g_n(x)$ is a concave parabola, we are interested in the positive solution $x_0$ for which $f_n(x_0)=g_n(x_0)$ is the largest possible, it is not hard to see that it happens when $x\thickapprox \sqrt{m}$.

\begin{corollary}\label{cor8}
For the digraph $\overrightarrow{C}_{n^2-n+ t}[\overrightarrow{C}_{n}]$ with $1\leq t \leq n$, we have that
\[\textrm{\emph{dac}}(\overrightarrow{C}_{n^2-n+ t}[\overrightarrow{C}_{n}])\leq n^{2}.\]
\end{corollary}
\begin{proof}
Consider $x=n$ and $x=n-1$, then $f_n(n)=n^2-n+ t$, $g_n(n)=n^2+n+1$ and $\min\{f_n(n),g_n(n)\}=f_n(n)=n^2-n+ t$.  Next, $f_n(n-1)=\left\lfloor n^2+ t +\tfrac{t}{n-1} \right\rfloor$, $g_n(n-1)=n^2$ and $\min\{f_n(n-1),g_n(n-1)\}=g_n(n-1)=n^2$. It follows that $\textrm{dac}(\overrightarrow{C}_{n^2-n+ t}[\overrightarrow{C}_{n}])\leq n^2$.
\end{proof}

In order to give a lower bound for $\text{dac}(\overrightarrow{C}_{n^2-n+ t}[\overrightarrow{C}_{n}])$, we use a similar coloring of $\overrightarrow{C}_{n^2-n}[\overrightarrow{C}_{n}]$ given in Corollary \ref{cor6} and the fact of a direct cycle has dichromatic number 2.

\begin{theorem}\label{teo9}
Let $n$ be a natural number such that $n\not=4,6$.  For the digraph $\overrightarrow{C}_{n^2-n+ t}[\overrightarrow{C}_{n}]$ with $1 \leq t \leq n$, 
\[\text{\emph{dac}}(\overrightarrow{C}_{n^2-n+ t}[\overrightarrow{C}_{n}])= n^{2}.\]
\end{theorem}
\begin{proof}
We can extend a diachromatic coloring $\varsigma$ of $\overrightarrow{C}_{n^2-n}$ with the colors $[n]$ and chromatic classes $\varsigma^{-1}(i)=\{x_1^i,x_2^i,\dots,x_{n-1}^i\}$ to a coloring $\varsigma'$ of $\overrightarrow{C}_{n^2-n+ t}$ as follows.
Without loss of generality, we suppose that $x^1_1x^2_1$ is the arc with the colors $1$ and $2$ of $\overrightarrow{C}_{n^2-n}$. Make $t$ subdivisions to $x^1_1x^2_1$ obtaining a directed path, i.e. $(x^1_1,x_0,x_1,\dots,x_{t-1},x^2_1)$. Color the vertices $\{x_0,x_1\dots,x_{t-1}\}$ with the alternating colors $2,1,2,1,\dots$ then the coloring is also acyclic and complete using $n$ colors.

Now, consider the complete graph $K_n$ with the vertex-set $\{v_1,v_2,\dots,v_n\}$. Take a factorization of $K_n$ into $n-1$ Hamiltonian cycles $H_j\cong \overrightarrow{C}_{n}$, for $j\in [n-1]$. We define the coloring $\varsigma_j\colon V(H_j)\rightarrow [n]$ such as $\varsigma_j(v_k)=k$. 
The digraph $\overrightarrow{C}_{n^2-n+ t}[\overrightarrow{C}_{n}]$ has vertices $(x_j^i,v_k)$ where $v_k\in V(H_j)$ and $(x_l,v_k)$, where $v_k\in V(H_1)$ if $l$ is odd and $v_k\in V(H_2)$ if $l$ is even. Color the vertices of $\overrightarrow{C}_{n^2-n+ t}[\overrightarrow{C}_{n}]$ with $n^2$ colors such that $(x_j^i,v_k)\mapsto (\varsigma(x_j^i),\varsigma_j(v_k))$, $(x_l,v_k)\mapsto (\varsigma'(x_l),\varsigma_1(v_k))$ if $l$ is odd and $(x_l,v_k)\mapsto (\varsigma'(x_l),\varsigma_2(v_k))$ if $l$ is even. The result follows due to Corollary \ref{cor8}.
\end{proof}

\subsection{On the harmonious chromatic number}

In this subsection, we bound the harmonious chromatic number of $D=\overrightarrow{C}_{m}[\overrightarrow{C}_{n}]$ for $m$ close to $n^2-n$. A close relationship between this subsection and the previous one can be observed. 

First, we improve the upper bound of Equation \ref{eq1} for $m=n^2-n-t$ following the idea of comparing two functions.

\begin{theorem}\label{teo10}
Let $m,n$ be natural numbers such that $m,n\geq 3$. For any harmonious coloring of $D=\overrightarrow{C}_{m}[\overrightarrow{C}_{n}]$ using $k$ colors
\[k\geq \min\left\{ \max\{f_n(x),g_n(x)\} \textrm{ with } x\in\mathbb{N}\right\}\] where $f_n(x)=\left\lceil mn/x\right\rceil$ and $g_n(x)=x(n+1)+1$.
\end{theorem}
\begin{proof}
Let $\varsigma\colon V(D) \rightarrow [k]$ be an harmonious  $k$-coloring of $D$.  Let $x= \max\{ \left|\varsigma^{-1}(i)\right| : i\in \left[k\right]\}$, that is, let $x$ be the cardinality of the largest chromatic class of $\varsigma$. Without loss of generality,  suppose that $x= \left|\varsigma^{-1}(k)\right|$. Since $\varsigma$ defines a partition of $V(D)$ it follows that $k\geq mn/x$ and then $k\geq f_n(x)$.

On the other hand, there are $x(n+1)$ vertices from any largest chromatic class, and then there are at least $x(n+1)+1$ chromatic classes,  it follows that $k\geq g_n(x)$. Thus, we have that $k\geq\max\{f_n(x),g_n(x)\}.$
Finally, we obtain that
\[k\geq \min\left\{ \max\{f_n(x),g_n(x)\} \textrm{ with } x\in\mathbb{N}\right\}.\]
\end{proof}

\begin{corollary}\label{cor11}
For the digraph $\overrightarrow{C}_{n^2-n-t}[\overrightarrow{C}_{n}]$ with $1\leq t < n$, we have that
\[h(\overrightarrow{C}_{n^2-n-t}[\overrightarrow{C}_{n}])\geq n^{2}.\]
\end{corollary}
\begin{proof}
Consider $x=n-1$ and $x=n-2$, then $f_n(n-1)=\left\lceil n^2-\tfrac{t}{n-1} \right\rceil$, $g_n(n-1)=n^2$ and $\max\{f_n(n-1),g_n(n-1)\}=g_n(n-1)=n^2$.  Next, $f_n(n-2)=\left\lceil n^2+n+1-\tfrac{t-2}{n-1} \right\rceil$, $g_n(n-2)=n^2-n-1$ and $\max\{f_n(n-2),g_n(n-2)\}=f_n(n-2)=\left\lceil n^2+n+1-\tfrac{t-2}{n-1} \right\rceil$. It follows that $h(\overrightarrow{C}_{n^2-n-t}[\overrightarrow{C}_{n}])\geq n^2$.
\end{proof}

In order to give an upper bound for $h(\overrightarrow{C}_{n^2-n-t}[\overrightarrow{C}_{n}])$, we color $\overrightarrow{C}_{n^2-n}[\overrightarrow{C}_{n}]$ using the technique given in Corollary \ref{cor6}, and a particular unfold of $K_n-A$ where $A$ is a particular set of edges.

\begin{theorem}\label{teo12}
Let $n$ be a natural number such that $n\not=4,6$.  For the digraph $\overrightarrow{C}_{n^2-n-t}[\overrightarrow{C}_{n}]$ with $1 \leq t < n$, we have that
\[h(\overrightarrow{C}_{n^2-n-t}[\overrightarrow{C}_{n}])= n^{2}.\]
\end{theorem}
\begin{proof}
Consider the graph $K_n-A$ where $A$ is a set of edges incident to a vertex $u\in V(K_n)$, with $|A| = t$.  Now, we obtain the graph $G$ which is an unfold $K_n-A$ such that each edge of $A$ is a leaf where its vertex of degree 1 is the corresponding vertex $u$ in $K_n-A$.

Next, we obtain the directed cycle $\overrightarrow{C}_{n^2-n-t}$ to unfold $G$ since $G$ it is an Eulerian digraph.  We can color $\overrightarrow{C}_{n^2-n-t}[\overrightarrow{C}_{n}]$ following the same technique of Theorem \ref{teo2} and obtaining our coloring.

The result follows due to Corollary \ref{cor11} and this upper bound. 
\end{proof}

\section{Recursive results}\label{section4}

The lexicographic product of digraphs $D$ and $H$ is an operation that produces a digraph $D[H]$,  and then we can obtain the digraph $(D[H])[H]$ and so on. We define	 $D[H]^{i}:=(D[H]^{i-1})[H]$ with $D[H]^1:=D[H]$. If $D\cong H$, we write $[H]^{i+1}$.

Note that Corollary \ref{cor4} produces a $k$-minimal digraph for which, their chromatic classes $\{v^l_{u_i,1},v^l_{u_i,2},\dots,v^l_{u_i,q}\}$ have cardinality equal to $q$, therefore this digraph and the relabel $H$-factorizations fulfills the hypothesis, hence, a recursive construction can be done given an initial digraph $D$ and an $H$-factorization.

\begin{corollary}\label{cor13}
Let $D$ be a $k$-minimal digraph of order $n$ with a balanced $k$-coloring, such that $qk=n$. If $K_m$ has a relabel $H$-factorization into $q$ factors, then $D[H]^i$ is $k^im$-minimal with a balanced  $k^im$-coloring, for all $i\in\mathbb{Z}^+$.
\end{corollary}

Hence, we can extend Corollary \ref{cor6} as follows.

\begin{corollary}\label{cor14}
Let $n$ be a natural number such that $n\not=4,6$. The digraph $D=\overrightarrow{C}_{n^2-n}[\overrightarrow{C}_n]^i$ is $n^{i+1}$-minimal with a balanced $n^{i+1}$-coloring. Hence, for all $i\in\mathbb{Z}^+$,
\[\text{\emph{dac}}(D)=h(D)=n^{i+1}.\]
\end{corollary}

Now, we determine the  dichromatic number of  $D=\overrightarrow{C}_{n^2-n}[\overrightarrow{C}_{n}]^i$.  

First, we adapt a result of \cite{MR1817491} in the following lemma, namely, Proposition 32 $(iii)$ and Proposition 34 together with  Corollary 43.  We omit the proof because it is analogous to the original one generalizing from transitive tournaments to acyclic digraphs.  

\begin{theorem}\label{teo15} 
Let $D,H$ be digraphs such that $D$ has order $m$ and $\text{\emph{dc}}(H)=k$. If $r$ is the maximum order  of an acyclic set of vertices of $D$,  then
\begin{itemize}
\item $\text{\emph{dc}}(D[H]) \geq \left \lceil \frac{k\cdot m }{r} \right\rceil.$

\item Moreover, if $D$ contains a spanning subdigraph isomorphic to a circulant digraph $\overrightarrow{C}_m(J)$ such that the induced subdigraph of the vertices $\{0,1,\dots,r\}$ is acyclic, then 
\[\text{\emph{dc}}(D[H]) = \left \lceil \frac{k\cdot m }{r} \right\rceil.\]
\end{itemize}
\end{theorem} 

Next, we use a result concerning the recurrence relation that appears in the solution of the well-known Josephus Problem,  see \cite{MR1397498,MR1108748} and see \cite{MR3711038} for an application in digraphs.  
	
\begin{theorem}\cite{MR1108748}\label{teo16}
Consider the recurrence relation $T_n(i)= \left \lceil \frac{n}{n-1} T_n(i-1)\right\rceil$ with $i\geq 1$ and $T_n(0)=1$. For each integer $n\geq 2$ there is  real number $c_n$ such that 
\begin{itemize}
\item $T_n(i)=c_n\left(\frac{n}{n-1}\right)^i+ 	e_{i,n}$ where $n\geq 4$ and $-n+2<e_{i,n}\leq 0$.
\item $T_3(i)=\left \lfloor c_3\left(\frac{3}{2}\right)^i\right\rfloor$ where $c_3\approx1.62227\dots$ is an irrational number.
\item $T_2(i)=2^i$.
\end{itemize}
\end{theorem} 

Now,  we can determine $\text{dc}([\overrightarrow{C}_{n}]^i)$ and then $\text{dc}(\overrightarrow{C}_{n^2-n}[\overrightarrow{C}_{n}]^i)$.
 
\begin{lemma}\label{lema17}
$\text{\emph{dc}}([\overrightarrow{C}_{n}]^i)=T_n(i)$.
\end{lemma}
\begin{proof}
Clearly, the maximal order of an acyclic set of vertices of $\overrightarrow{C}_{n}$ is $n-1$, and $\overrightarrow{C}_{n}$ contains an isomorphic copy of a circulant digraph over $\mathbb{Z}_n$ as a spanning subdigraph.
By Theorem \ref{teo15}, it follows that $\text{dc}(\overrightarrow{C}_{n}[\overrightarrow{C}_{n}])=\left \lceil \frac{2n}{n-1} \right\rceil$ and since $\left \lceil \frac{n}{n-1}  \right \rceil=2$, thus 
\[\text{dc}(\overrightarrow{C}_{n}[\overrightarrow{C}_{n}])=\left \lceil \frac{n}{n-1}\left \lceil \frac{n}{n-1} \right\rceil \right\rceil.\]
Repeating this argument $i-1$ times, it follows that
Theorem \ref{teo16} completes the proof.
\end{proof}

As a consequence of Theorem \ref{teo15}  and Lemma \ref{lema17} it follows that
	
\begin{theorem}\label{teo18}
$\text{\emph{dc}}(\overrightarrow{C}_{n^2-n}[\overrightarrow{C}_{n}]^i)= \left\lceil\frac{n^2-n}{n^2-n-1}T_n(i) \right\rceil.$ 
\end{theorem} 

\section{Final remarks}

Different factorizations can be considered, for instance, in \cite{MR0236063} answered the question about a factorization of $K_{n}$ into Hamiltonian directed paths, that is, if the elements of some group of order $n$ can be arranged in a sequence $c_1,c_2,\dots,c_n$ such that $c_1c_2c_3\dots c_i\not=c_1c_2c_3\dots c_j$ whenever $i\not= j$. This is shown to be possible for any Abelian group with exactly one element of order 2 and for the non-Abelian group of order 21.  Then, we have the following corollary.

\begin{corollary}\label{cor19}
If $K_n$ accepts a $\overrightarrow{P}_{n}$-factorization, then the digraph $D=\overrightarrow{C}_{n(n+1)}[\overrightarrow{P}_{n}]^i$ is $n(n+1)^i$-minimal with a balanced $n(n+1)^i$-coloring. Hence		
\[\text{dac}(D)=h(D)=n(n+1)^i,\]
for all $i\in\mathbb{Z}^+$.
\end{corollary}

In \cite{MR4221847} was given $H$-factorizations of the complete graphs via quadratic residues where $H$ is a circulant graph.  There are several types of factorizations of the complete graphs,  for instance, see \cite{MR1617664,MR3812023}.

Another possible interesting problem is due to Theorem \ref{teo18} and Corollary \ref{cor14}.  For any pair of positive integers $i,n$,  we have the dichromatic numbers of $\overrightarrow{C}_{n^2-n}[\overrightarrow{C}_{n}]^i$. Although these results provide an infinite number of pairs  of integers  $a\leq b$ such that there exists an oriented graph $D$ satisfying that $\text{dc}(D)=a$ and $\text{dac}(D)=b$. For a given $a\leq b$, there exists a oriented graph $D$ such that $\text{dc}(D)=a$ and $\text{dac}(D)=b$? It is known that the question for graphs is answered by Bhave, see \cite{MR532949}.

Finally, we remark that for graphs, an elementary identification is called as a harmonious homomorphism, see \cite{AMOR}. For digraphs, we need to extend proper coloring to digraphs. In this paper, we use the extension given by Hedge and Castelino \cite{MR2895432,MR3329642} and Edwards \cite{MR2998438}. The other possibility is via acyclic coloring. We define the \emph{harmonious dichromatic number} $\text{dh}(D)$ of $D$ as the smallest $k$ for which there exists an acyclic and harmonious coloring of $D$ using $k$ colors. Clearly, $\text{dh}(D)\leq h(D)$ but the lower bound of Equation \ref{eq1} of $h(D)$ is not necessarily a lower bound of $\text{dh}(D)$.


\bibliographystyle{plain}
\bibliography{biblio}

\end{document}